\documentclass[twoside, a4paper, nofootinbib,10pt]{article}
\usepackage{amssymb}
\usepackage{IEEEtrantools}
\usepackage[mathscr]{eucal}
\usepackage[dvips]{graphicx}
\usepackage{amsmath}
\usepackage{amsthm}
\usepackage{pstricks}
\usepackage{caption}
\usepackage{cite}
\def \ni{\noindent}

\newcommand{\be}{\begin{equation}}
\newcommand{\ee}{\end{equation}}
\newcommand{\ben}{\begin{equation*}}
\newcommand{\een}{\end{equation*}}
\newcommand{\bes}{\begin{eqnarray}}
\newcommand{\ees}{\end{eqnarray}}
\newcommand{\besn}{\begin{IEEEeqnarray*}{rCl}}
\newcommand{\eesn}{\end{IEEEeqnarray*}}

\newcommand{\txt}{\textrm}

\newtheorem{theorem}{Theorem}
\newtheorem*{theorem*}{Theorem}
\newtheorem{lemma}{Lemma}
\newtheorem*{definition*}{Definition}
\newtheorem*{lemma*}{Lemma}
\newtheorem*{prop*}{Proposition}
\newtheorem*{corollary*}{Corollary}

\DeclareFontFamily{U}{mathx}{}
\DeclareFontShape{U}{mathx}{m}{n}{<-> mathx10}{}
\DeclareSymbolFont{mathx}{U}{mathx}{m}{n}
\DeclareMathAccent{\widecheck}{0}{mathx}{"71}

\title{On the Construction of Euclidean Invariant and Reflection Positive Measures on the Cylindrical Compactification of Distributions}
\author{T. Tlas}
\date{}

\begin{document}
\maketitle
\thispagestyle{empty}

\begin{abstract}
\ni A simple construction of Euclidean invariant and reflection positive measures on the cylindrical compactification is performed under a weaker hypothesis than has recently been obtained. Moreover, the results are extended to the case when the theory under consideration has finitely many constraints.
\end{abstract}

\section*{Introduction}

One of the main underlying mathematical goals of constructive quantum field theory is obtaining Euclidean invariant and reflection positive measures on the space of distributions. This is a subject with a long history and a great deal of knowledge has been acquired \cite{jaffe, pphi, strocci}. It has been shown recently \cite{construction}, that there is a rather simple rigorous construction of Euclidean invariant and reflection positive measures on a compactification of the space of distributions.\footnote{This doesn't mean that the difficulties of this subject totally vanish. Rather, they shift to showing that the resulting measure has nontrivial support away from the `corona' set. See \cite{herglotz} for an illustration of this phenomenon in a simple setting.} Unfortunately, the proof of the translation invariance part of the Euclidean invariance relied on a post hoc condition on the constructed measure, notably that its first moment be finite. This brings us to the aim of the current work, which is twofold: to recover the results of \cite{construction} without requiring this awkward condition, and to show how easily the construction can be extended to cover the case of a theory with constraints. This is not a purely academic exercise, since some of the most interesting theories are simply constrained Gaussians, Yang-Mills being the most notable example \cite{yang1, yang2, yang3, yang4}. \newline

Before we begin, let us list our choices of notation and conventions. All our functions and distributions will be defined on $\mathbb{R}^D$ where $D$ is fixed but arbitrary. Moreover, they will all be $\mathbb{R}^K$-valued, where $K$ is also fixed but arbitrary.\footnote{All the results below trivially extend to the case when the range is taken to be $\mathbb{C}^K$.} The dot product of two vectors $u$ and $v$ is denoted by $uv$. $B(x,r)$ stands for the open ball in $\mathbb{R}^D$ of center $x$ and radius $r$. For an open set $O \subset \mathbb{R}^D$, $C^\infty(O)$ stands for the space of smooth functions with supports in $O$. $C^\infty(\mathbb{R}^d) \equiv C^\infty$. For any two sets $A, B$, $A - B$ stands for the complement of $B$ in $A$ (not to be confused with the Minkowski difference). $\mathfrak{E}$ stands for the Euclidean group. Our Fourier transform convention is $\widehat{f}(p) = \int e^{- i p x} f(x) dx$. $H^\alpha(O)$ stands for the Sobolev space of order $\alpha$ defined to be the closure of $C^\infty_0(O)$ with respect to the norm $||f ||_{H^\alpha} = \int (1 + p^2)^\alpha \widehat{f}(p) dp$. Finally, the notation $f(x) \lesssim g(x)$ means that $f(x) \leq c g(x)$ where $c$ is a constant independent of $x$, but which may depend on the other parameters in the problem under consideration. \newline

\section*{The Unconstrained Case}

Let us now briefly describe the general setup. While it closely parallels that of \cite{construction}, it differs from it in that the infrared regularization is given by introducing a cutoff into the interaction Lagrangian, as opposed to dealing with fields defined on the sphere. The reason for doing this is that it turns out that eventually, apart from superficially aesthetic reasons, there is no real advantage of putting things on the sphere, while there are significant technical complications of doing so in the case of gauge theory. This is because in the approach advocated here, one would need to deal with measures on spaces of paths on a sphere as opposed to those on a flat space \cite{yang1, yang2, yang3, yang4}. This can be done, and there is a well-developed theory for that \cite{stroock}, but it comes at a significant - and ultimately unnecessary in our case - technical cost. Additionally, it is precisely the new regularization which will allow us to recover the results of \cite{construction} under a weaker hypothesis. \newline

Therefore, consider the expression

\be
\label{eq:cov}
C(f,g) = \frac{1}{(2 \pi)^D} \int_{\mathbb{R}^D} \frac{ \overline{\widehat{f}(p)}  \widehat{g}(p)  }{p^2 + 1} dp.
\ee

It follows by standard methods \cite{jaffe, bogachev} that this defines a Gaussian measure $d\mu$ on the space of distributions $\Phi'$.\footnote{The support of the measure is known (see e.g. \cite{support}) to be significantly smaller than the full space $\Phi'$. As a crude estimate, one can easily see that (\ref{eq:cov}) defines a trace-class, bilinear form on the space $\frac{1}{(1+x^2)^D} H^\alpha$ for a sufficiently large $\alpha$. Therefore, it follows that the measure is supported on its dual, i.e. on $(1 + x^2)^D H^{- \alpha}$. We won't need this more refined control of the support of the measure $\mu$ in what follows.} \newline

Now, let $\sigma$ be a smooth, positive, rotationally invariant function supported in $B(0,1)$ whose integral is equal to 1. For $\Lambda >0$ let $\sigma_\Lambda(x) = \Lambda^D \sigma(\Lambda x)$. For any $\phi \in \Phi'$ we have that $\phi_\Lambda = \phi \ast \sigma_\Lambda \in C^\infty$. Now, if $\mathcal{L}$ is a bounded measurable function on $\mathbb{R}^l$, we have that 

\ben
\int_{B(0,r)} \mathcal{L} \big ( \phi_\Lambda(x), \nabla^2 \phi_\Lambda(x), \dots, (\nabla^2)^l \phi_\Lambda (x) \big ) dx
\een

is a well-defined, bounded function on $\Phi'$. Consider

\be
\label{eq:expr}
\frac{    \int F[\phi] e^{- \int_{B(0,r)} \mathcal{L} \big ( \phi_\Lambda(x),  \dots, (\nabla^2)^l \phi_\Lambda (x) \big ) dx    }   d \mu  }{    \int e^{- \int_{B(0,r)} \mathcal{L} \big ( \phi_\Lambda(x), \dots, (\nabla^2)^l \phi_\Lambda (x) \big ) dx   }   d \mu  } .
\ee

This is perfectly well-defined for any $F \in L^1(\mu)$, however, as discussed in \cite{construction}, it is convenient to restrict the class of functions under consideration to cylindrical ones. A function $F$ on $\Phi'$ is cylindrical if there are $g_1, \dots, g_k \in C^\infty_0$ such that $F[\phi] = f( g_1(\phi) , \dots, g_k(\phi)   )$, where $f$ is a bounded continuous function on $\mathbb{R}^k$. \newline

Now, the expression (\ref{eq:expr}) is a regularized, in the ultraviolet and in the infrared, version of the informal expression

\be
\label{eq:heuristic}
\frac{ \int F[\phi] e^{ - \int_{\mathbb{R}^D} (\phi  (\nabla^2 + 1) \phi) + \mathcal{L}(\phi(x), \dots, (\nabla^2)^l \phi(x)) dx    }  d \phi  }{   \int e^{ - \int_{\mathbb{R}^D} (\phi  (\nabla^2 + 1) \phi) + \mathcal{L}(\phi(x), \dots, (\nabla^2)^l \phi(x)) dx    }  d \phi   } .
\ee

Of course we would like to remove both cutoffs. Therefore, we need to consider an entire sequence of expressions of the form (\ref{eq:expr}), but with $\Lambda, r$ (and possibly even $l$) being sequences, with $\Lambda_n, r_n \to \infty$. We want to extract a single number from this sequence. This is most conveniently done via a Banach limit $L$.\newline

Putting everything together, we have

\be
\label{eq:int}
L \Bigg (  \frac{    \int F[\phi] e^{- \int_{B(0,r_n)} \mathcal{L}_n \big ( \phi_{\Lambda(_n}x),  \dots, (\nabla^2)^l \phi_{\Lambda_n} (x) \big ) dx    }   d \mu  }{    \int e^{- \int_{B(0,r_n)} \mathcal{L}_n \big ( \phi_{\Lambda_n}(x), \dots, (\nabla^2)^l \phi_{\Lambda_n} (x) \big ) dx   }   d \mu  }  \Bigg ) \equiv I(F).
\ee

We can now state the following 

\begin{theorem}
There is a unique compactification $\mathring{\Phi}'$ of $\Phi'$ such that every cylindrical function $F$ has a unique extension $\mathring{F}$ to $\mathring{\Phi}'$. Addtionally, for any sequence $\{ \mathcal{L}_n\}_{n=1}^\infty$, there is a unique, rotationally invariant probability measure $\nu$ on $\mathring{\Phi}'$ such that $$I(F) = \int \mathring{F} d\nu .$$ Moreover, one can choose the sequences $\{r_n\}_{n=1}^\infty$ and $\{\Lambda_n\}_{n=1}^\infty$ such that $\nu$ is reflection positive and Euclidean invariant.
\end{theorem}

Before we proceed with the proof of the theorem, we would like to draw attention to a fact that will play a role later. In the definition of (\ref{eq:int}) we've required $\mathcal{L}$ to be bounded. It is obvious however, that the expression (\ref{eq:int}) is perfectly well-defined, as long as $\mathcal{L}$ is semibounded, i.e. bounded from below. Nonetheless, requiring boundedness of $\mathcal{L}$ does not imply any loss of generality as, on one hand, one expects that any universality class contains such Lagrangians, and on the other hand, one has the following

\begin{prop*}
For any three sequences $\{r_n\}_{n=1}^\infty, \{\Lambda_n\}_{n=1}^\infty$ and $\{ \mathcal{L}_n\}_{n=1}^\infty$, with the $\mathcal{L}_n$'s being semibounded, there is a sequence $\{ \widetilde{\mathcal{L}}_n \}_{n=1}^\infty$ of bounded Lagrangians such that (\ref{eq:int}) gives the same value for all $F$'s after replacing $\{\mathcal{L}_n\}_{n=1}^\infty$ with $\{\widetilde{\mathcal{L}}_n \}_{n=1}^\infty$.
\end{prop*}

\begin{proof}[Proof of proposition]
Since (\ref{eq:int}) is invariant under $\mathcal{L}_n \to \mathcal{L}_n + c_n$ for any sequence $\{c_n\}_{n=1}^\infty$, we can assume that $\mathcal{L}_n \geq 0$ for all $n$. Now, let $\widetilde{\mathcal{L}}_n$ be equal to $\frac{\mathcal{L}_n  }{\epsilon_n \mathcal{L}_n + 1}$, where $\epsilon_n > 0$ is chosen such that 

\besn
\frac{ \bigg | \int e^{- \int_{B(0,r_n)} \widetilde{\mathcal{L}}_n \big ( \phi_{\Lambda_n}(x), \dots \big ) dx   }   d \mu   -    \int e^{- \int_{B(0,r_n)} \mathcal{L}_n \big ( \phi_{\Lambda_n}(x), \dots \big ) dx   }   d \mu    \bigg |
}
{   \int e^{- \int_{B(0,r_n)} \mathcal{L}_n \big ( \phi_{\Lambda_n}(x), \dots \big ) dx   }   d \mu   } &= & \\
  \frac{  \int \Big ( e^{- \int_{B(0,r_n)} \widetilde{\mathcal{L}}_n \big ( \phi_{\Lambda_n}(x), \dots \big ) dx   }   -    e^{- \int_{B(0,r_n)} \mathcal{L}_n \big ( \phi_{\Lambda_n}(x), \dots \big ) dx   }   \Big ) d \mu    
}
{   \int e^{- \int_{B(0,r_n)} \mathcal{L}_n \big ( \phi_{\Lambda_n}(x), \dots \big ) dx   }   d \mu   }      & < &  \frac{1}{n}.
\eesn

The fact that this can be done follows at once from dominated convergence, since $ \frac{x}{\epsilon x + 1} \to x$ for every $x$ as $\epsilon \to 0$. Therefore, we have that

\besn
\Bigg | \frac{ \int F[\phi] e^{- \int_{B(0,r_n)} \widetilde{\mathcal{L}}_n \big ( \phi_{\Lambda_n}(x), \dots \big ) dx   }  d\mu      }{\int   e^{- \int_{B(0,r_n)} \widetilde{\mathcal{L}}_n \big ( \phi_{\Lambda_n}(x), \dots \big ) dx   }       d \mu   }      -    \frac{\int F[\phi] e^{- \int_{B(0,r_n)} \mathcal{L}_n \big ( \phi_{\Lambda_n}(x), \dots \big ) dx   } }{e^{- \int_{B(0,r_n)} \mathcal{L}_n \big ( \phi_{\Lambda_n}(x), \dots \big ) dx   } } \Bigg | & \leq &  \\
\Bigg | \frac{ \int F[\phi] e^{- \int_{B(0,r_n)} \widetilde{\mathcal{L}}_n \big ( \phi_{\Lambda_n}(x), \dots \big ) dx   }  d\mu      }{\int   e^{- \int_{B(0,r_n)} \widetilde{\mathcal{L}}_n \big ( \phi_{\Lambda_n}(x), \dots \big ) dx   }       d \mu   }      -    \frac{\int F[\phi] e^{- \int_{B(0,r_n)} \mathcal{L}_n \big ( \phi_{\Lambda_n}(x), \dots \big ) dx   } }{e^{- \int_{B(0,r_n)} \widetilde{ \mathcal{L}}_n \big ( \phi_{\Lambda_n}(x), \dots \big ) dx   } } \Bigg | & + & \\
\Bigg | \frac{ \int F[\phi] e^{- \int_{B(0,r_n)} \mathcal{L}_n \big ( \phi_{\Lambda_n}(x), \dots \big ) dx   }  d\mu      }{\int   e^{- \int_{B(0,r_n)} \widetilde{\mathcal{L}}_n \big ( \phi_{\Lambda_n}(x), \dots \big ) dx   }       d \mu   }      -    \frac{\int F[\phi] e^{- \int_{B(0,r_n)} \mathcal{L}_n \big ( \phi_{\Lambda_n}(x), \dots \big ) dx   } }{e^{- \int_{B(0,r_n)} \mathcal{L}_n \big ( \phi_{\Lambda_n}(x), \dots \big ) dx   } } \Bigg | & \leq & \\
\frac{ 2 ||F||_{L^\infty}        }{  n- 1        }   .
\eesn

The result now follows since the Banach limit in (\ref{eq:int}) extends the usual limit.
\end{proof}

Although, as we have just seen, there is no loss of generality in using bounded Lagrangians, there is one disadvantage. It is that this implies that in general one will not be able to perform the integrals explicitly, since one cannot use polynomial expressions for $\mathcal{L}_n$ as these are all unbounded.\footnote{This is the reason why in the physics literature, where practicality of calculations is paramount, bounded Lagrangians are rarely used.} Of course, in general, the integrals are not doable explicitly anyway, still, there \textit{is} one important exception when they are, which is when the Lagrangians are quadratic. We will deal with such a case later below.

\begin{proof}[Proof of Theorem 1]
The existence of the compactification with the stated properties and of the aforementioned measure carry over essentially verbatim from the corresponding proof in \cite{construction} and will not be repeated here. \newline

To deal with Euclidean invariance and reflection positivity, we need a couple of lemmas.

\begin{lemma}
Let $A \subset \mathbb{R}^D$ be bounded and $T \in \mathfrak{E}$. Then, for any cylindrical function $F$, we have
\ben
\int F_T[\phi] e^{ - \int_A \mathcal{L}(\phi_\Lambda)    } d\mu[\phi] = \int F[\phi] e^{ - \int_{T(A)} \mathcal{L}(\phi_\Lambda)   } d\mu[\phi].
\een
\end{lemma}

\begin{proof}[Proof of Lemma 1] Since $\int e^{i f(\phi)} d\mu[\phi] = e^{- C(f,f)}$, it follows at once from the Euclidean invariance of (\ref{eq:cov}) that $\int e^{i Tf(\phi)}  d\mu[\phi] = \int e^{i f(\phi)} d\mu[\phi]$. Since the set of trigonometric polynomials is $L^1(\mu)$ dense in the set of cylindrical functions, we get that 
\be
\label{eq:invariance}
\int F_T[\phi] d\mu[\phi] = \int F[\phi] d\mu[\phi]
\ee
 \newline

Now, let $\sum_A \mathcal{L}(\phi_\Lambda) \Delta x$ stand for a Riemann sum approximation to $\int_A \mathcal{L}(\phi_\Lambda)$ with $\sum_{T(A)} \mathcal{L}(\phi_\Lambda) \Delta x$ denoting the corresponding approximation to $\int_{T(A)} \mathcal{L}(\phi_\Lambda)$. Since $e^{-\sum_A \mathcal{L}(\phi_\Lambda) \Delta x}$ and $e^{-\sum_{T(A)} \mathcal{L}(\phi_\Lambda) \Delta x}$ are both cylindrical, we have, in view of (\ref{eq:invariance}), that

\ben
\int F_T[\phi] e^{-\sum_A \mathcal{L}(\phi_\Lambda) \Delta x} d\mu[\phi] = \int F[\phi] e^{-\sum_{T(A)} \mathcal{L}(\phi_\Lambda) \Delta x} d\mu[\phi].
\een
Applying now dominated convergence, we get what we want.
\end{proof}

Before we proceed with the second lemma, note that if $T$ is a rotation then $T(B(0,r)) = B(0,r)$ from which the rotational invariance of $\nu$ follows at once. \newline

\begin{lemma}
Let $O_1, \dots, O_m \subset \mathbb{R}^D$ be open and pairwise disjoint. Then, there are subspaces $\Phi'_{O_1}, \dots, \Phi'_{O_m}$ and $\Phi'_{(O_1 \cup \dots \cup O_m)^c}$ of $\Phi'$, such that $\Phi'_A$ consists of elements supported\footnote{Support is of course understood in the distributional sense.} in $\overline{A}$. Moreover, there are Gaussian measures $\mu_{O_1}, \dots, \mu_{O_m}$ and $\mu_{(O_1 \cup \dots \cup O_m)^c}$, such that $\mu_{A}$ is supported on $\Phi'_A$, and these measures satisfy 

\be
\label{eq:conv}
\mu = \mu_{O_1} \ast \dots \ast \mu_{O_m} \ast \mu_{(O_1 \cup \dots \cup O_m)^c}.
\ee

\end{lemma}

\begin{proof}[Proof of Lemma 2]
This is nothing but a manifestation of the Markovian property of the Gaussian field \cite{nelson, dimock}. Note that the Cameron-Martin space of the measure $\mu$ is the Sobolev space $H^1(\mathbb{R}^D)$. It will be more convenient for us to use the fact that this space is unitarily equivalent to $H^{-1}(\mathbb{R}^D)$ via the map $(- \nabla^2 + 1)$, since the inner product in $H^{-1}$ coincides with the covariance of the measure. Now, it is straightforward to show (see Lemma 1 in \cite{dimock}) that we have the orthogonal decomposition 

\besn
H^{-1}(\mathbb{R}^D) & = & (\nabla^2 + 1) H^1(O_1) \oplus (- \nabla^2 + 1) H^1(O_2) \oplus \dots \\
& & \dots \oplus (- \nabla^2 + 1) H^1(O_m)  \oplus H^{-1}( ( O_1 \cup \dots \cup O_m)^c).
\eesn

It is a basic result\footnote{See e.g. Corollary 3.3.3. in \cite{bogachev}. } that an orthogonal decomposition of the Cameron-Martin spaces as above induces measures $\mu_{O_1}, \dots, \mu_{O_m}$ and $\mu_{(O_1\cup \dots \cup O_m)^c}$ which satisfy (\ref{eq:conv}). \newline

It only remains to show that the supports of these measures are as indicated above. Therefore, introduce the following inner product

\ben
\langle f , g \rangle = \int_{\mathbb{R}^D} ( 1 + x^2 )^{2 \alpha} \Big ( ( - \nabla^2 + 1) f  \Big )   \Big (  (-\nabla^2 + 1) g   \Big ),
\een

and denote the corresponding norm by $||| \cdot |||$. For any set $A$, let $H(A)$ stand for the completion of $C^\infty_0(A)$ in the norm $||| \cdot |||$.  It is straightforward to see\footnote{This is essentially the same computation as in \cite{support}.} that for sufficiently large $\alpha$ and $\beta$, the natural inclusion $H(A) \to H^{-1}(A)$ is Hilbert-Schmidt.  This in turn implies \cite{support1} that the support of the measure whose covariance is $H^{-1}(A)$ is contained in the dual of $H(A)$. Since this dual is naturally identified with a collection of distributions supported on $A$, we get what we want.\end{proof}

Let us now deal with reflection positivity. Let $\delta  > 0$. Define

\besn
\Pi^+_\delta  =  \{ (x_1, \dots, x_D) \in \mathbb{R}^D \, : \, x_D > \delta \}  & , &   \Pi^-_\delta  =  \{ (x_1, \dots, x_D) \in \mathbb{R}^D \, : \, x_D < - \delta \} \\
B^+_{\delta}(0,r)   = B(0,r) \cap \Pi^+_\delta \,&  , & \, B^-_{\delta}(0,r)   =  B(0,r) \cap \Pi^-_\delta \\
\txt{and} \quad B^0_{\delta}(0,r) & = & B(0,r) - \big ( B^+_{\delta}(0,r) \cup B^-_{\delta}(0,r) \big )
\eesn

We apply now Lemma 2 to the case when $O_1 = \Pi^+_\delta$ and $O_2 = \Pi^-_\delta$. To reduce clutter, we will denote $\Phi_{\Pi^+_\delta}, \Phi_{\Pi^-_\delta}$ and $\Phi_{(\Pi^+_\delta \cup \Pi^-_\delta)^c}$ by $\Phi^+, \Phi^-$ and $\Phi^0$ respectively. \newline

Now, suppose that $g_1, \dots, g_k$ are in $C^\infty_0(\Pi^+_0)$ and that $F$ is a cylindrical function of the form $f( g_1(\phi), \dots, g_k(\phi)  )$. Let $\Theta$ be the reflection in the $D$-th coordinate (i.e. $\Theta(x_1, \dots, x_D) = (x_1, \dots, -x_D)$). Set $\delta = \frac{1}{\Lambda}$. Note that for all sufficiently large $\Lambda$'s, we have that $g_1, \dots, g_k \in C_0^\infty(\Pi^+_{2\delta})$. \newline

Putting everything together, we have

\bes
& & \int F_\Theta [\phi] F[\phi] e^{- \int_{ B^+_{2 \delta}(0,r)  \cup B^-_{2 \delta}(0,r)  }    \mathcal{L}(\phi_\Lambda, \dots)         } d \mu[\phi] \nonumber \\
& = & \int F_\Theta[\phi^-] F[\phi^+] e^{- \int_{B^+_{2 \delta}(0,r)   }   \mathcal{L}(\phi^+_\Lambda, \dots)  }   e^{- \int_{B^-_{2 \delta}(0,r)   }   \mathcal{L}(\phi^-_\Lambda, \dots)  } d\mu^+[\phi^+] d\mu^-[\phi^-] \nonumber \\
& = & \bigg (  \int F[\phi] e^{- \int_{ B^+_{2 \delta} (0,r) } \mathcal{L}(\phi_\Lambda)  } d\mu[\phi]    \bigg ) \bigg (    \int F_\Theta[\phi] e^{- \int_{ B^-_{2 \delta} (0,r) } \mathcal{L}(\phi_\Lambda)  } d\mu[\phi]  \bigg ) \\
\label{eq:positive}
& = & \bigg (  \int F[\phi] e^{- \int_{ B^+_{2 \delta} (0,r) } \mathcal{L}(\phi_\Lambda)  } d\mu[\phi]    \bigg )^2 \geq 0
\ees

Now, as mentioned in the proof of the proposition above, we can assume that $\mathcal{L} \geq 0$. Letting $M = \sup_{\mathbb{R}^l}(\mathcal{L})$, we have that

\bes
\nonumber
\bigg |  e^{-  \int_{ B^+_{2 \delta}(0,r)  \cup B^-_{2 \delta}(0,r)  }    \mathcal{L}(\phi_\Lambda, \dots)     } - e^{-   \int_{ B(0,r) }    \mathcal{L}(\phi_\Lambda, \dots)  }  \bigg | & = & \\
\nonumber
e^{-   \int_{ B(0,r) }    \mathcal{L}(\phi_\Lambda, \dots)  }  \bigg | e^{+  \int_{ B^0_{2\delta}(0,r) }    \mathcal{L}(\phi_\Lambda, \dots)     }    -1       \bigg | & \lesssim & \\
\nonumber
e^{-   \int_{ B(0,r) }    \mathcal{L}(\phi_\Lambda, \dots)  } \int_{ B^0_{2\delta}(0,r) }    \mathcal{L}(\phi_\Lambda, \dots)   & \lesssim & \\
\label{eqns:ineq}
e^{-   \int_{ B(0,r) }    \mathcal{L}(\phi_\Lambda, \dots)  }  ( M r^{D-1} \delta )  = e^{-   \int_{ B(0,r) }    \mathcal{L}(\phi_\Lambda, \dots)  } \Big  ( \frac{M r^{D-1}}{\Lambda}  \Big ) .
\ees

Therefore, we see that

\besn
\Bigg |  \frac{ \int F[\phi] F_\Theta[\phi] \bigg (    e^{-  \int_{ B^+_{2 \delta}(0,r)  \cup B^-_{2 \delta}(0,r)  }    \mathcal{L}(\phi_\Lambda, \dots)     } - e^{-   \int_{ B(0,r) }    \mathcal{L}(\phi_\Lambda, \dots)  }        \bigg )      d \mu         }{  \int   e^{-   \int_{ B(0,r) }    \mathcal{L}(\phi_\Lambda, \dots)  }         d\mu      }    \Bigg | & \leq & \\
|| f||^2_{L^\infty} \Big ( \frac{M r^{D-1}}{\Lambda}  \Big ).
\eesn 

Thus, if the sequences $\{r_n\}_{n=1}^\infty$ and $\{\Lambda_n \}_{n=1}^\infty$ are chosen to satisfy

\be
\label{eq:cond}
\frac{ M_n r_n^{D-1} }{ \Lambda_n   } \to 0,
\ee

we have, in view of (\ref{eq:positive}), that $I(F F_\Theta) \geq 0$. Thus, $\mu$ is reflection positive. \newline

It remains to show translation invariance. Let $t \in \mathfrak{E}$ be a translation. Let $F[\phi] = f( g_1(\phi), \dots, g_k(\phi))$ be a cylindrical function, and once again, set $\delta = \frac{1}{\Lambda}$. Note that since the supports of $g_1, \dots, g_k$ are compact, then for a sufficiently large $r$, all of them will be contained in $B(0,r) \cap t ( B(0,r)) $. If we apply Lemma 2 now, where there is a single $O_1 = B(0,r) \cap t(B(0,r))$, we get that\footnote{The reader should perhaps draw the relevant sets for the case $D = 2$.}

\besn
\int F[\phi] e^{ - \int_{ (B(0,r) - t(B(0,r+\delta)))  \cup ( B(0,r- \delta) \cap t (B(0,r- \delta))       )       } \mathcal{L}(\phi_\Lambda)          }  d\mu[\phi] & = & \\
\int F[\phi]  e^{ - \int_{ (B(0,r) - t(B(0,r+\delta)))        } \mathcal{L}(\phi_\Lambda)          }  e^{ - \int_{  ( B(0,r- \delta) \cap t (B(0,r- \delta))       )       } \mathcal{L}(\phi_\Lambda)          }    d\mu[\phi] & = & \\
\bigg ( \int_{\Phi'_{  B(0,r) \cap t(B(0,r))  }   }  F([\psi])  e^{ - \int_{  ( B(0,r- \delta) \cap t (B(0,r- \delta))       )       } \mathcal{L}(\psi_\Lambda)          }  d\mu_{  B(0,r) \cap t(B(0,r))    }[\psi] \bigg ) \times & \dots  & \\ 
\dots \times \bigg (  \int_{  \Phi'_{  (B(0,r) \cap t(B(0,r)) )^c }    }      e^{ - \int_{ (B(0,r) - t(B(0,r+\delta)))  }    \mathcal{L}(\chi_\Lambda)    }  d\mu_{ (B(0,r) \cap t(B(0,r)) )^c   }[\chi] \bigg ) & = & \\
\bigg ( \int_{\Phi'_{  B(0,r) \cap t(B(0,r))  }   }  F([\psi])  e^{ - \int_{  ( B(0,r- \delta) \cap t (B(0,r- \delta))       )       } \mathcal{L}(\psi_\Lambda)          }  d\mu_{  B(0,r) \cap t(B(0,r))    }[\psi] \bigg ) \times & \dots  & \\ 
\dots \times \bigg (  \int_{  \Phi'_{  (   B(0,r) \cap t(B(0,r)) )^c }    }      e^{ - \int_{ t( B(0,r)  )  - B(0, r + \delta)   }    \mathcal{L}(\chi_\Lambda)    }  d\mu_{ (B(0,r) \cap t(B(0,r)) )^c   }[\chi] \bigg ) &= & \\
\int F[\phi] e^{ - \int_{ (t( B(0,r)    ) - B(0, r+ \delta) )   \cup (  B(0,r- \delta) \cap t (B(0,r- \delta))      )       } \mathcal{L}(\phi_\Lambda)          }  d\mu[\phi] ,& &
\eesn

where the penultimate equality follows from Lemma 1, since there is\footnote{Explicitly, $T$ is a reflection in the hyperplane orthogonal to the translation vector defining $t$ and passing through its midpoint.} a Poincar\'{e} element $T$, such that $$T \Big ( B(0,r) - t(B(0,r+\delta))) \Big ) =  t( B(0,r)  )  - B(0, r + \delta)    . $$

Now, proceeding as in the reflection positivity proof, we obtain that

\besn
\bigg |   e^{ - \int_{ (B(0,r) - t(B(0,r+\delta)))  \cup ( B(0,r- \delta) \cap t (B(0,r- \delta))       )       } \mathcal{L}(\phi_\Lambda)          }  - e^{- \int_{B(0,r)} \mathcal{L} (\phi_\Lambda)}   \bigg | &\lesssim& \\
  e^{- \int_{B(0,r)} \mathcal{L} (\phi_\Lambda)}   \Big ( \frac{M r^{D-1}}{\Lambda} \Big ),
\eesn

and that

\besn
\bigg |   e^{ - \int_{ (t( B(0,r)    ) - B(0, r+ \delta) )   \cup (  B(0,r- \delta) \cap t (B(0,r- \delta))      )           } \mathcal{L}(\phi_\Lambda)          }  - e^{- \int_{t( B(0,r))} \mathcal{L} (\phi_\Lambda)}   \bigg | &\lesssim& \\
e^{- \int_{t( B(0,r))} \mathcal{L} (\phi_\Lambda)} \Big (  \frac{M r^{D-1}}{\Lambda} \Big ).
\eesn

Since we know from Lemma 1 that 

\ben
 \int F_t[\phi] e^{ - \int_{B(0,r)} \mathcal{L}(\phi_\Lambda)} d\mu = \int F[\phi]  e^{ - \int_{t(B(0,r))} \mathcal{L}(\phi_\Lambda)} d\mu,
\een

we get

\besn
\Bigg |  \frac{ \int F[\phi] \bigg ( e^{ - \int_{B(0,r)} \mathcal{L}(\phi_\Lambda)} -   e^{ - \int_{t(B(0,r))} \mathcal{L}(\phi_\Lambda)}   \bigg )       d\mu[\phi]  }{ \int e^{- \int_{B(0,r)} \mathcal{L}(\phi_{\Lambda})     }  d\mu[\phi]  }   \Bigg | \lesssim || F||_{L^\infty} \Big (  \frac{M r^{D-1}}{\Lambda} \Big ) .
\eesn

It thus follows that if we take sequences that satisfy (\ref{eq:cond}), we will get $I(F_t) = I(F)$, and the proof is complete.
\end{proof}

\section*{The Constrained Case}

We would now like to extend the work above to the case when the integration in the informal expression (\ref{eq:heuristic}) is not over the entire set of fields, but is restricted to a subset of it. More precisely, we will be interested in integrating over the subset of fields satisfying finitely many constraints of the form

\bes
\nonumber
\kappa_1(\phi(x), \nabla^2 \phi(x), \dots, (\nabla^2)^{m_1} \phi(x)) = 0, & &\\
\nonumber
\kappa_2(\phi(x), \nabla^2 \phi(x), \dots, (\nabla^2)^{m_2} \phi(x)) = 0, & & \\
\label{eqs:constraints}
\vdots \qquad \qquad \qquad  \, \,  \, & & \\
\nonumber
\txt{and} \quad \kappa_\rho(\phi(x), \nabla^2 \phi(x), \dots, (\nabla^2)^{m_\rho} \phi(x)) = 0 & &
\ees

where these constraints are understood to hold at every point of the Euclidean spacetime. As a typical example, the reader should keep in mind the examples $\phi^2(x) - 1 =0$ of the vectorial nonlinear sigma model, or $\phi^\dagger (x) \phi(x) - \mathbb{I} = 0$ of the principal chiral one. \newline

Of course, we immediately face the problem that constraints as above are, almost surely, ill-defined on the domain of integration, for it is a standard fact \cite{jaffe} that the subset of functions in $\Phi'$ has measure zero. There are two possible avenues one can choose to proceed on. The first one is to press on, trying to make sense of the constraints, as they are on $\Phi'$. Alternatively, we could try imposing the constraint in the regularized integral, i.e. inside the brackets on the left hand side of (\ref{eq:int}). Let us analyze the first option.\newline

How should the constraint equations be interpreted on the space $\Phi'$? We do know what the constraints mean on $C_0^\infty$ functions, which are a dense subset of the space of integration. We therefore can say that if $\tilde{K}$ is the set of $C_0^\infty$ functions which satisfy the constraint, then its closure $K = \overline{\tilde{K}}$ is the set of elements of $\Phi'$ which do so as well. This is actually a reasonable definition since in the case of a linear constraint - that is in the case when the constraint is of the form $\mathcal{D} \phi = 0 $, where $\mathcal{D}$ is a linear, differential operator with constant coefficients - it follows easily that $\phi \in K \iff \mathcal{D} \phi = 0$, where the derivatives are interpreted in the distributional sense.\newline

Given that we now have a subset of our measure space, albeit one of measure zero, and are trying to restrict our measure to it, we can proceed in a similar way to that discussed in \cite{al}. In other words, we let $K_\delta$ be the `thickening' of $K$ (to use the terminology of \cite{al}). Then, for a function $F$ we take the Banach limit of $$\frac{\int_{K_\delta} F[\phi] d\mu[\phi] }{\mu(K_\delta)}$$ as $\delta \to 0$. This procedure does indeed generate a measure supported on $K$; however, there are several shortcomings:

\begin{itemize}
\item[i-] It is rather difficult to work with this definition and perform computations.
\item[ii-] It is not clear how to demonstrate Euclidean invariance and reflection positivity.
\item[iii-] Since the space $\Phi'$ is not metrizable, there is no obvious choice of `thickenings' $K_\delta$ to pick. One can try to bypass this issue by using a smaller space as a support for $\mu$, in particular a Hilbert space (see the footnote after the expression for the covariance of the measure $\mu$). Unfortunately, the natural metric on this support is not Euclidean invariant which makes it even less clear how could one demonstrate Euclidean invariance.
\item[iv-] This is perhaps the most serious one. The constructed measure depends solely on the constraint \textit{subset} and not on the \textit{functional form} of the constraints. This is certainly not what we expect from the formal physics calculations since changing the functional form of the constraint amounts to changing the action via the Jacobian factors.
\end{itemize}

In view of the above, we shall not pursue this approach further and will instead concentrate on implementing the constraint in a regularized form in a manner mimicking closely what is done in the physics literature. Note that we need to impose the constraints (\ref{eqs:constraints}) at every point of spacetime. Therefore, we would like to insert into our functional integral an expression of the form

\ben
\prod_{i = 1}^\rho \bigg (  \prod_{x \in \mathbb{R}^D}  \delta \Big ( \kappa_i (\phi(x), \dots \Big  )     \bigg ) .
\een

We thus have to make sense of two things: first, of the delta function in the functional integral context, and second, of the product of such functions over the set of all points of spacetime. While one can make sense of a Gaussian analog of the space of distributions (and thus of the delta function), thus solving the first problem, this would be an overkill in this context. Besides, a more primitive approach turns out to solve the second problem as well. \newline

The most pedestrian approach to the delta function is to interpret any expression with it via a limit. The choice of the limiting sequence is not unique, but it turns out that the distributional equation

\be
\label{eq:delta}
\delta(x) = \frac{1}{\sqrt{\pi}} \lim_{a \to \infty} \sqrt{a}  e^{- a x^2}
\ee

is most convenient for our purposes. Note that 

\besn
\delta(x_1) \dots \delta(x_r) & = &  \frac{1}{(\sqrt{\pi})^\rho} \lim_{a_1, \dots, a_\rho \to \infty} \sqrt{a_1 \dots a_\rho} e^{- a_1 x_1^2 - \dots - a_r x_\rho^2}  \\
& = & \frac{1}{(\sqrt{\pi})^\rho} \lim_{a \to \infty} \sqrt{a^\rho} e^{- a (x_1^2 + \dots + x_\rho^2)},
\eesn

where again the equalities hold in the distributional sense. Therefore, focusing on the case of a single constraint for now, we can make sense of an expression of the form $\prod_{x} \delta( \kappa[\phi(x)])$ by equating it up to an overall constant, with $e^{- a \sum_x \kappa^2\phi[(x)]}$. The overall constant can easily be recovered by normalizing this to have integral one. If we now rescale $a$ and interpret the sum over all points of spacetime as an integral, we are led to consider $e^{- a \int \kappa^2\phi[(x)]}$.\newline

In view of the above, we see that we can implement a collection of constraints of the form (\ref{eqs:constraints}) by simply adding the term 

\ben
a \int \kappa_1^2[\phi(x)] + \dots + \kappa_\rho^2[\phi(x)]
\een

to the action and sending $a \to \infty$.\newline

Putting everything together, we can now define the constrained analogue of (\ref{eq:int}) via

\be
\label{eq:sigma}
L \bigg ( \frac{ \int F[\phi] e^{- a_n \int_{B(0,r_n)}  \kappa_1^2[\phi_{\Lambda_n}] + \dots + \kappa^2_\rho [\phi_{\Lambda_n}] }      d\mu[\phi] }{   \int  e^{- a_n \int_{B(0,r_n)}  \kappa_1^2[\phi_{\Lambda_n}] + \dots + \kappa^2_\rho [\phi_{\Lambda_n}] }  d\mu[\phi]    }   \bigg ) \equiv I(F) ,
\ee

where it is understood that the sequences $\{r_n\}_{n=1}^\infty, \{ \Lambda_n \}_{n=1}^\infty$ and $\{a_n \}_{n=1}^\infty$ all go to infinity.\newline

We now have the following immediate corollary to Theorem 1 above\newline

\begin{corollary*} For constraints of the form (\ref{eqs:constraints}) which are bounded, the expression (\ref{eq:sigma}) defines a Euclidean invariant and reflection positive measure on a compactification of the space of distributions.
\end{corollary*}

Of course, there is no loss of generality in assuming that constraints are bounded, as one can simply replace a given constraint $\kappa$ with e.g. $\frac{\kappa}{1+ \kappa^2}$, or more generally, $\tau(\kappa)$ for any other convenient function $\tau$ such that $\tau(x) \sim x$ around $x=0$. In effect, this simply amounts to a different choice of the representation of the delta function in (\ref{eq:delta}).\newline

At this point, an objection may be raised to the effect that how do we know that (\ref{eq:sigma}) does give the rigorous analogue of the constrained theory? After all, the constraints themselves are never imposed, only a softened version of them. Moreover, we cannot do better, if we want to use the theorem above to recover Euclidean invariance and reflection positivity, i.e. we cannot impose the constraints sharply \textit{before} removing the ultraviolet cutoff. This is because making the constraints sharper means taking a bigger $a_n$ in (\ref{eq:sigma}). But this implies that the bound $M_n$ in (\ref{eq:cond}) on the Lagrangian will grow, which in turn forces us to increase $\Lambda_n$. \newline

Answering this objection is rather difficult since the only answer which will be fully convincing would amount to fully constructing the model. There is some circumstantial evidence in that imposing the constraint softly does produce a model with desired properties \cite{kopper}. What we are going to do below is to verify that we do recover the expected result in the computable case of local, linear constraints imposed on a Gaussian. Therefore, we're going to consider the case where the constraints are of the form 

\be
\label{eq:linear}
\kappa_1[\phi] = \mathcal{D}_1[\phi], \dots, \kappa_\rho[\phi] = \mathcal{D}_\rho[\phi],
\ee

where $\mathcal{D}_1, \dots, \mathcal{D}_\rho$ are linear, differential operators with constant coefficients. This is not an entirely physically irrelevant set up since it does for example cover the case of abelian Yang-Mills.\newline

Now, as discussed above, we can easily interpret constraining a Gaussian theory with linear constraints, since the constraint equations easily generalize to the distributional setting. We simply restrict our covariance (\ref{eq:cov}) to those test functions which satisfy the constraint equations. This covariance then defines a Gaussian measure on the dual of the restricted space. \newline 

More precisely, let $\mathcal{D}_1(p), \dots, \mathcal{D}_\rho(p)$ be the Fourier transforms of the differential operators describing the constraints. Note that each one of the $\mathcal{D}$'s is a $p$-valued linear transformation. Let $\Pi(p)$ denote the ($p$-dependent) orthogonal projection on the intersection of the kernels of these transformations. The expression

\be
\label{eq:covr}
C^\kappa(f,g) = \frac{1}{(2 \pi)^D} \int_{\mathbb{R^D}} \frac{ \overline{\Big ( \Pi(p) \widehat{f}(p) \Big ) }   \Big ( \Pi(p) \widehat{g}(p) \Big ) }{p^2 + 1} dp
\ee

defines the covariance of the constrained theory. We now have the following

\begin{theorem}
For any sequences $\{a_n\}_{n=1}^\infty$ and $\{\Lambda_n \}_{n=1}^\infty$ going to infinity, one can choose a sequence $\{r_n\}_{n=1}^\infty$ going to infinity such that the measure given by (\ref{eq:sigma}) for the constraints (\ref{eq:linear}) coincides with (\ref{eq:covr}).
\end{theorem}

\begin{proof}
Begin by choosing two sequences $\{a_n\}_{n=1}^\infty$ and $\{\Lambda_n\}_{n=1}^\infty$ diverging to infinity. Let $\{e_k\}_{k=1}^\infty$ stand for an orthonormal basis of $L^2(\mathbb{R}^D)$ consisting of Schwarz functions (e.g. take products of normalized Hermite functions). Let $V_m$ stand for the intersection of $B(0,m)$ in $L^2$ with the span of $\{e_1, \dots, e_m\}$. Note that $V_1 \subset V_2 \subset \dots$ and that $V = \cup_{m=1}^\infty V_m$ is dense in $L^2$.\newline

We now rely heavily on the fact that all involved measures in this case will be Gaussian.\footnote{A good background reference for the facts used in the proof is 9.3 in \cite{jaffe}.} First, we use that 

\ben
\frac{   e^{- a_n \int_{B(0,r_n)}  \kappa_1^2[\phi_{\Lambda_n}] + \dots + \kappa_\rho^2[\phi_{\Lambda_n}]   } d\mu[\phi]}{ \int e^{- a_n \int_{B(0,r_n)}  \kappa_1^[\phi_{\Lambda_n}] + \dots + \kappa_\rho^2[\phi_{\Lambda_n}]   } d\mu[\phi] }
\een

defines a Gaussian measure. Let us denote its covariance by $C_{a_n, \Lambda_n, r_n}$. Now, take any monotone increasing sequence $\{r_m \}_{m=1}^\infty$ with $r_m \nearrow \infty$. Since, for a fixed $a_n$ and $\Lambda_n$, the quadratic forms

\ben
Q_{a_n, \Lambda_n, r_m}(\phi) = \int_{\mathbb{R}^D} \phi \big (-  \nabla^2 + 1) \phi + a_n \int_{B(0,r_m)} \kappa_1^2[\phi_{\Lambda_n}] + \dots + \kappa_\rho^2[\phi_{\Lambda_n}]
\een

are closed on the Sobolev space $H^1$ and satisfy $Q_{a_n, \Lambda_n, r_m} \leq Q_{a_n, \Lambda_n, r_{m+1}}$, we have (see e.g. Theorem S.14 in \cite{simon}) that the corresponding self-adjoint operators converge in the strong resolvent sense. It thus follows that

\ben
\lim_{r_n \to \infty} C_{a_n, \Lambda_n, r_n}(f,g) = C_{a_n, \Lambda_n, \infty}(f,g),
\een

where $C_{a_n, \Lambda_n, \infty}(f,g)$ is given by the formula

\be
\label{eq:inter1}
 \frac{1}{(2 \pi)^D} \int _{\mathbb{R}^D}    \overline{ \widehat{f}(p)   }   \Big (      (p^2 + 1) \mathbb{I} + a_n  \widehat{\sigma}^2_{\Lambda_n} \big ( \mathcal{D}_1^\dagger(p) \mathcal{D}_1(p) + \dots + \mathcal{D}_\rho^\dagger(p) \mathcal{D}_\rho(p)   \big )    \Big )^{-1}   \widehat{g}(p)   dp.
\ee

It thus follows that the Gaussian measures whose covariances are $C_{a_n, \Lambda_n, r_n}$ converge weakly, as $r_n \to \infty$, to the Gaussian measure whose covariance is given by $C_{a_n, \Lambda_n, \infty}$.\newline

Now, for a given $m$, choose $r_m$ such that 

\be
\label{eq:inter2}
\Big |C_{a_m, \Lambda_m, r_m}(f,g) - C_{a_m, \Lambda_m, \infty}(f,g) \Big  | < \frac{1}{m},
\ee

for all $f,g \in V_m$.\newline

Pick $f, g \in V$. Consider $C_{a_n, \Lambda_n, \infty}(f,g)$. Since $\mathcal{D}_1^\dagger(p) \mathcal{D}_1(p) + \dots + \mathcal{D}_\rho^\dagger(p) \mathcal{D}_\rho(p) $ is a self-adjoint matrix, it can be unitarily diagonalized, and thus, the inverse in (\ref{eq:inter1}) can be easily computed. Sending then $a_n, \Lambda_n \to \infty$ and applying dominated convergence, we see that 

\ben
C_{a_n, \Lambda_n, \infty}(f,g) \underset{a_n, \Lambda_n \to \infty}{\longrightarrow} C^\kappa(f,g)
\een

since the components orthogonal to the kernel of $\mathcal{D}_1^\dagger(p) \mathcal{D}_1(p) + \dots + \mathcal{D}_\rho^\dagger(p) \mathcal{D}_\rho(p) $ are suppressed by the diverging denominator. Combining this with (\ref{eq:inter2}), we have that for all $f, g \in V$ 

\ben
C_{a_m, \Lambda_m, r_m} (f,g) \underset{m \to \infty}{\longrightarrow} C(f,g). 
\een

Since $V$ is dense in $L^2$, and since $C_{a_m, \Lambda_m, r_m}(f,g), C_{a_m, \Lambda_m, \infty}(f,g)$ and $C^\kappa(f,g)$ are all bounded from above by $\frac{1}{(2 \pi)^D}\int \overline{\widehat{f}(p)}  \widehat{g}(p) dp$, it follows that $$C_{a_m, \Lambda_m, r_m} ( \cdot, \cdot) \to C^\kappa ( \cdot, \cdot)$$ weakly on $L^2$, and thus, weakly as bilinear forms on the Schwarz space.\newline 

It consequently follows that the corresponding Gaussian measures converge weakly. This, together with the fact that on convergent sequences, the Banach limit coincides with the usual one, implies that the measure given by (\ref{eq:sigma}) coincides with the Gaussian one whose covariance is given by (\ref{eq:covr}); and this is exactly what we want.
\end{proof}

 \textbf{Acknowledgments:} The author would like to thank J. Merhej for reading a preliminary version of this paper and for the numerous comments which have greatly improved its readability.

\texttt{{\footnotesize Department of Mathematics, American University of Beirut, Beirut, Lebanon.}
}\\ \texttt{\footnotesize{Email address}} : \textbf{\footnotesize{tamer.tlas@aub.edu.lb}}

\end{document}